\documentclass[11pt]{article}

\usepackage{geometry}
\usepackage{amssymb}
\usepackage{amsthm}
\usepackage{amsmath}
\usepackage{enumerate}
\usepackage{booktabs}
\usepackage{graphicx}
\usepackage{multirow}
\usepackage{mathtools}
\usepackage{color}
\usepackage{algorithm,algorithmic}
\usepackage{authblk}
\usepackage{appendix}
\usepackage{hyperref}

\newtheorem{theorem}{Theorem}
\newtheorem{corollary}{Corollary}
\newtheorem{lemma}{Lemma}

\newtheorem{definition}{Definition}
\newtheorem{property}{Property}
\newtheorem{remark}{Remark}

\newcommand{\R}{\mathbb{R}}
\newcommand{\C}{\mathbb{C}}

\title{Optimal unit triangular factorization of symplectic matrices}

\author[1]{Pengzhan Jin}
\author[2,3]{Zhangli Lin}
\author[4,5]{Bo Xiao}
\affil[1]{School of Mathematical Sciences, Peking University, Beijing 100871, China.}
\affil[2]{Department of Mathematics, Yau Mathematical Sciences Center, Tsinghua University, Beijing 100084, China.}
\affil[3]{Yanqi Lake Beijing Institute of Mathematical Sciences and Applications, Beijing 101408, China.}
\affil[4]{Academy of Mathematics and Systems Science, Chinese Academy of Sciences, Beijing 100190, China.}
\affil[5]{School of Mathematical Sciences, University of Chinese Academy of Sciences, Beijing 100049, China.}
\date{}

\begin{document}

\maketitle
\begin{abstract}
We prove that any symplectic matrix can be factored into no more than 5 unit triangular symplectic matrices, moreover, 5 is the optimal number. This result improves the existing triangular factorization of symplectic matrices which gives proof of 9 factors. We also show the corresponding improved conclusions for structured subsets of symplectic matrices. This factorization further provides an unconstrained optimization method on $2d$-by-$2d$ real symplectic group (a $2d^2+d$-dimensional Lie group) with $2d^2+3d$ parameters.
\end{abstract}

\section{Introduction}
Consider matrices over field $F$, and denote the $d$-by-$d$ identity matrix by $I_{d}$, let
\begin{equation*}
J\coloneqq\begin{bmatrix} 0 & I_{d} \\ -I_{d} & 0 \end{bmatrix},
\end{equation*}
which satisfies $J^{-1}=J^{T}=-J$.
\begin{definition} \label{def:SP}
A matrix $H\in F^{2d\times 2d}$ is called symplectic if $H^{T}JH=J$.
\end{definition}
We denote the collection of symplectic matrices by
\begin{equation*}
Sp(d,F)=\{H\in F^{2d\times 2d}|H^{T}JH=J\},
\end{equation*}
which forms a group, i.e., matrix symplectic group. For convenience, we also denote $\mathcal{SP}=Sp(d,F)$ when there are no arguments over $d$ and $F$. It should be noted that there is a special case for complex entries.
\begin{definition}
A matrix $H\in\C^{2d\times 2d}$ is called conjugate symplectic if $H^*JH=J$.
\end{definition}
Accordingly, we denote the collection of complex conjugate symplectic matrices by
\begin{equation*}
\mathcal{SP}^*=\{H\in \C^{2d\times 2d}|H^{*}JH=J\}.
\end{equation*}
Since the complex conjugate symplectic group is not an algebraic group over $\C$ , so the behavior of conjugate symplectic matrices is quite different from general symplectic matrices, which is always an algebraic group, thus in this work we mainly discuss general symplectic matrices. The group $\mathcal{SP}$ is important from both the pure mathematical point of view \cite{fomenko1995symplectic}, and the point of view of applications. For instance, it 
appears in classical mechanics and Hamiltonian dynamical systems \cite{abraham1978foundations,arnol2013mathematical,arnold2007mathematical}. The symplectic matrices also arise in the symplectic integrators which are essential for physical problems preserving the symplectic structure \cite{feng2010symplectic,hairer2006geometric,sanz2018numerical}. Moreover, it is also applied to linear control \cite{fassbender2007symplectic,mehrmann1991autonomous}, optimal control \cite{hassibi1999indefinite}, the theory of parametric resonance \cite{iakubovich1975linear}, as well as machine learning \cite{bondesan2019learning,jin2020sympnets}. The applications motivate the development of factorizations of symplectic matrices, such as the Iwasawa decomposition \cite{benzi2007iwasawa}, the QR-like factorization \cite{bunse1986matrix}, the polar factorization \cite{higham2006symmetric}, the SVD-like factorization \cite{xu2003svd}, and the transvections factorization \cite{flaschka1991analysis,mackey2003determinant}.

The most elementary decomposition units we here care about are the unit triangular symplectic matrices:
\begin{equation*}
    \mathcal{T}=\left\{\begin{bmatrix}
    I & S \\ 0 & I
    \end{bmatrix}\ {\rm or}\ 
    \begin{bmatrix}
    I & 0 \\ S & I 
    \end{bmatrix}\Bigg|S^T=S,S\in F^{d\times d}\right\},
\end{equation*}
denote
\begin{equation*}
    \mathcal{L}_{n}=\left\{H_n\cdots H_2H_1|H_i\in \mathcal{T},i=1,2,\cdots,n\right\}.
\end{equation*}
It is clear that $\mathcal{L}_{m}\subset \mathcal{L}_{n}\subset \mathcal{SP}$ for all integers $1\leq m\leq n$. \cite{kopeiko1978stabilization} has pointed out that $\bigcup_n\mathcal{L}_n=\mathcal{SP}$, i.e., any symplectic matrix can be written as the product of some unit triangular symplectic matrices. The current unit triangular factorizaiton of symplectic matrices with a finite upper bound is shown in \cite{jin2020unit} as follows.
\begin{theorem} \label{thm:fac_thm}
$\mathcal{SP}=\mathcal{L}_{9}$.
\end{theorem}
The above theorem indicates that any symplectic matrix can be written as the product of 9 unit triangular symplectic matrices. However, the number ``9'' is in fact not optimal. We will point out that any symplectic matrix can be factored into no more than 5 unit triangular symplectic matrices, and 5 is optimal. Below we only consider the case $d\geq 2$, since the result for case $d=1$ is trivial but slightly different. Especially, any symplectic matrix can be optimally factored into 4 unit triangular symplectic matrices for case $d=1$, one may easily verify it. 
\section{Preliminaries}
We first show some basic properties as well as recent new discoveries regarding to symplectic matrices. Property \ref{pro:smat_pro}-\ref{pro:LDU} are well-known and easy to verify, readers are referred to \cite{de2006symplectic,feng2010symplectic} for more details. Note that although Theorem \ref{thm:fac_thm}-\ref{thm:unit_ULU} are proved for $F=\R$ in \cite{jin2020unit}, in fact the proofs hold for any general field $F$. For convenience, we denote all the unimportant blocks in matrices by ``$\star$'' throughout this paper.
\begin{property} \label{pro:smat_pro}
If $H=\begin{bmatrix} A_{1} & B_{1} \\ A_{2} & B_{2} \end{bmatrix}\in F^{2d\times 2d}$ is a symplectic matrix and $A_{i},B_{i}\in F^{d\times d}$, then (i) $A_{1}^{T}A_{2}=A_{2}^{T}A_{1}$, (ii) $B_{1}^{T}B_{2}=B_{2}^{T}B_{1}$, (iii) $A_{1}^{T}B_{2}-A_{2}^{T}B_{1}=I$,
vice versa.
\end{property}

\begin{property} \label{pro:smat_cond}
If $S,P,Q\in F^{d\times d}$, then
\begin{enumerate}[(i)]
    \item The matrix $\begin{bmatrix} I & S \\ 0 & I \end{bmatrix}$ is symplectic if and only if $S^{T}=S$,
    \item The matrix $\begin{bmatrix} I & 0 \\ S & I \end{bmatrix}$ is symplectic if and only if $S^{T}=S$,
    \item The matrix $\begin{bmatrix} P & 0 \\ 0 & Q \end{bmatrix}$ is symplectic if and only if $Q=P^{-T}$.
\end{enumerate}
\end{property}
\begin{property}[LDU factorization] \label{pro:LDU}
If $H=\begin{bmatrix} A_{1} & B_{1} \\ A_{2} & B_{2} \end{bmatrix}\in F^{2d\times 2d}$ is a symplectic matrix and $A_{i},B_{i}\in F^{d\times d}$, moreover $A_{1}$ is nonsingular, then $H$ has three unique factorizations
\begin{equation} \label{eq:left_LDU}
\left\{
\begin{aligned}
&H=\begin{bmatrix} P_{1} & 0 \\ 0 & P_{1}^{-T} \end{bmatrix}\begin{bmatrix} I & 0 \\ S_{1} & I \end{bmatrix}\begin{bmatrix} I & T_{1} \\ 0 & I \end{bmatrix} \\
&S_{1}=A_{1}^{T}A_{2},\ T_{1}=A_{1}^{-1}B_{1},\ P_{1}=A_{1} \\
\end{aligned}
\right.,
\end{equation}
\begin{equation} \label{eq:center_LDU}
\left\{
\begin{aligned}
&H=\begin{bmatrix} I & 0 \\ S_{2} & I \end{bmatrix}\begin{bmatrix} P_{2} & 0 \\ 0 & P_{2}^{-T} \end{bmatrix}\begin{bmatrix} I & T_{2} \\ 0 & I \end{bmatrix} \\
&S_{2}=A_{2}A_{1}^{-1},\ T_{2}=A_{1}^{-1}B_{1},\ P_{2}=A_{1} \\
\end{aligned}
\right.,
\end{equation}
\begin{equation} \label{eq:right_LDU}
\left\{
\begin{aligned}
&H=\begin{bmatrix} I & 0 \\ S_{3} & I \end{bmatrix}\begin{bmatrix} I & T_{3} \\ 0 & I \end{bmatrix}\begin{bmatrix} P_{3} & 0 \\ 0 & P_{3}^{-T} \end{bmatrix} \\
&S_{3}=A_{2}A_{1}^{-1},\ T_{3}=B_{1}A_{1}^{T},\ P_{3}=A_{1} \\
\end{aligned}
\right.,
\end{equation}
where $S_{1},S_{2},S_{3},T_{1},T_{2},T_{3}$ are symmetric and $P_{1},P_{2},P_{3}$ are nonsingular.
\end{property}

\begin{theorem} \label{thm:nonsin_thm}
For any symplectic matrix $H\in F^{2d\times 2d}$, there exists a symmetric $S\in F^{d\times d}$ such that, the factorization
\begin{equation*}
    H=\begin{bmatrix} I & \lambda S \\ 0 & I \end{bmatrix}\begin{bmatrix} P_{\lambda} & \star \\ \star & \star \end{bmatrix}
\end{equation*}
holds with a nonsingular $P_{\lambda}$ for all $\lambda\neq 0$. Hence any symplectic matrix can be decomposed into a unit upper triangular symplectic matrix and a symplectic matrix with nonsingular left upper block. Furthermore, if needed, $S$ can be set to
\begin{equation*}
S=P\begin{bmatrix} O_{r} & 0 \\ 0 & I_{d-r} \end{bmatrix}P^{T}
\end{equation*}
when the left upper block $A_{1}$ of $H$ with $rank\ r$ is decomposed as
\begin{equation*}
A_{1}=P\begin{bmatrix} I_{r} & 0 \\ 0 & O_{d-r} \end{bmatrix}Q
\end{equation*}
where $P,Q\in F^{d\times d}$ are nonsingular.
\end{theorem}
\begin{theorem}[unit ULU factorization] \label{thm:unit_ULU}
For any symplectic matrix $H\in F^{2d\times 2d}$, there exist symmetric $S,T,U\in F^{d\times d}$ and a nonsingular $P\in F^{d\times d}$ such that
\begin{equation*}
H=\begin{bmatrix} I & S \\ 0 & I \end{bmatrix}\begin{bmatrix}P & \star \\ \star & \star \end{bmatrix}=\begin{bmatrix} I & S \\ 0 & I \end{bmatrix}\begin{bmatrix} I & 0 \\ T & I \end{bmatrix}\begin{bmatrix} I & U \\ 0 & I \end{bmatrix}\begin{bmatrix} P & 0 \\ 0 & P^{-T} \end{bmatrix}.
\end{equation*}
Furthermore, $T,U,P$ are uniquely determined by $H$ and $S$.
\end{theorem}
All of the above results still hold for conjugate symplectic matrices as long as we replace the transpose $T$ by conjugate transpose $*$. 

\section{Main results}
We begin with deducing a similar result to Theorem \ref{thm:nonsin_thm}, which transforms a symplectic matrix into one having nonsingular left upper block by a very simple symplectic matrix.
\begin{theorem}\label{thm:nonsin_fac}
For any symplectic matrix $H\in F^{2d\times 2d}$, there exist $\delta_1,\cdots,\delta_d\in\{0,1\}$, and a nonsingular $P\in F^{d\times d}$ such that
\begin{equation*}
H=\begin{bmatrix} I & {\rm diag}(\delta_1,\cdots,\delta_d) \\ 0 & I \end{bmatrix}\begin{bmatrix}P & \star \\ \star & \star \end{bmatrix}.
\end{equation*}
\end{theorem}
\begin{proof}
Assume that $H=\begin{bmatrix} A & B \\ C & D\end{bmatrix}$, $A=[A_1^T,\cdots,A_d^T]^T$, $C=[C_1^T,\cdots,C_d^T]^T$ for $A_i,C_i\in F^{1\times d}$. Without loss of generality, suppose that $\{A_1,\cdots,A_r\}$ is a maximal linearly independent subset of $\{A_1,\cdots,A_d\}$, $r={\rm rank}(A)$. Consider $\{A_1,\cdots,A_r,A_{r+1}-C_{r+1},\cdots,A_d-C_d\}$, if 
\begin{equation*}
    \lambda_1A_1+\cdots+\lambda_rA_r+\lambda_{r+1}(A_{r+1}-C_{r+1})+\cdots+\lambda_d(A_d-C_d)=0
\end{equation*}
for $\lambda_i\in F$, then
\begin{equation*}
    \eta_1A_1+\cdots+\eta_rA_r-\lambda_{r+1}C_{r+1}-\cdots-\lambda_dC_d=0
\end{equation*}
for some $\eta_i\in F$. Since the complementary bases theorem \cite[Theorem 3.1]{dopico2006complementary} points out that $\{A_1,\cdots,A_r,C_{r+1},\cdots,C_d\}$ is linearly independent, we know $\eta_1=\cdots=\eta_r=\lambda_{r+1}=\cdots=\lambda_d=0$, thus $\lambda_1=\cdots=\lambda_d=0$, i.e., $\{A_1,\cdots,A_r,A_{r+1}-C_{r+1},\cdots,A_d-C_d\}$ is linearly independent. Consequently
\begin{equation*}
    H=\begin{bmatrix} I & {\rm diag}(\delta_1,\cdots,\delta_d) \\ 0 & I \end{bmatrix}\begin{bmatrix}P & \star \\ \star & \star \end{bmatrix}
\end{equation*}
is the desired decomposition, where $\delta_1=\cdots=\delta_r=0$, $\delta_{r+1}=\cdots=\delta_d=1$, $P=[A_1^T,\cdots,A_r^T,(A_{r+1}-C_{r+1})^T,\cdots,(A_d-C_d)^T]^T$ is nonsingular.
\end{proof}
This theorem also holds for conjugate symplectic case.

\subsection{Optimal unit triangular factorization}
\begin{lemma}\label{lem:nonsin_utf}
If $H=\begin{bmatrix}A & B \\ C & D\end{bmatrix}\in \mathcal{SP}$ and $A\in F^{d\times d}$ is nonsingular, then there exist symmetric $S,T,U,V\in F^{d\times d}$ such that
\begin{equation*}
    H=\begin{bmatrix} I & 0 \\ S & I \end{bmatrix}\begin{bmatrix} I & T \\ 0 & I \end{bmatrix}\begin{bmatrix} I & 0 \\ U & I \end{bmatrix}\begin{bmatrix} I & V \\ 0 & I \end{bmatrix}\in \mathcal{L}_4.
\end{equation*}
\end{lemma}
\begin{proof}
According to Property \ref{pro:LDU} (equation (\ref{eq:center_LDU})) and the fact that any square matrix can be factored as the product of two symmetric square matrices \cite{taussky1972role}, we know there exist symmetric $R,W\in F^{d\times d}$ and symmetric nonsingular $P_1,P_2\in F^{d\times d}$ such that 
\begin{equation*}
H=\begin{bmatrix} I & 0 \\ R & I \end{bmatrix}\begin{bmatrix} (P_1P_2) & 0 \\ 0 & (P_1P_2)^{-T} \end{bmatrix}\begin{bmatrix} I & W \\ 0 & I \end{bmatrix},
\end{equation*}
then we can readily check that 
\begin{equation*}
H=\begin{bmatrix} I & 0 \\ R+P_1^{-1}P_2^{-1}P_1^{-1}-P_1^{-1} & I \end{bmatrix}\begin{bmatrix} I & P_1 \\ 0 & I \end{bmatrix}\begin{bmatrix} I & 0 \\ P_2-P_1^{-1} & I \end{bmatrix}\begin{bmatrix} I & W-P_2^{-1} \\ 0 & I \end{bmatrix}\in \mathcal{L}_4.
\end{equation*}
\end{proof}

\begin{remark}
Since the determinant of conjugate symplectic matrix is not necessarily 1, so the above lemma does not hold for complex conjugate case. But we can prove that if $H=\begin{bmatrix}A & B \\ C & D\end{bmatrix}\in \mathcal{SP}^*$ and $A\in \C^{d\times d}$ is nonsingular and similarity to a real matrix (for example, $H$ is Hermitian positive definite conjugate symplectic), then $A$ can be written as the product of two Hermitian matrices \cite{frobenius1917sitzungsber}, therefore $H$ can also be unit-triangular factored as 4 factors.
\end{remark}

\begin{theorem}\label{thm:utf}
For any symplectic matrix $H\in F^{2d\times 2d}$, there exist $\delta_1,\cdots,\delta_d\in\{0,1\}$ and symmetric $S,T,U,V\in F^{d\times d}$ such that
\begin{equation}\label{eq:utf}
H=\begin{bmatrix} I & {\rm diag}(\delta_1,\cdots,\delta_d) \\ 0 & I \end{bmatrix}\begin{bmatrix} I & 0 \\ S & I \end{bmatrix}\begin{bmatrix} I & T \\ 0 & I \end{bmatrix}\begin{bmatrix} I & 0 \\ U & I \end{bmatrix}\begin{bmatrix} I & V \\ 0 & I \end{bmatrix}\in \mathcal{L}_5.
\end{equation}
Furthermore, $\mathcal{L}_4\subsetneq \mathcal{SP}=\mathcal{L}_5$.
\end{theorem}
\begin{proof}
By Theorem \ref{thm:nonsin_fac} and Lemma \ref{lem:nonsin_utf}, we obtain equation (\ref{eq:utf}) and consequently $\mathcal{SP}=\mathcal{L}_5$. On the other hand, choose an asymmetric and nonsingular matrix $Q\in F^{d\times d}$. If there exist symmetric $S,T,U,V$ such that
\begin{equation*}
\begin{bmatrix} I & 0 \\ S & I \end{bmatrix}\begin{bmatrix} I & T \\ 0 & I \end{bmatrix}\begin{bmatrix} I & 0 \\ U & I \end{bmatrix}\begin{bmatrix} I & V \\ 0 & I \end{bmatrix}=\begin{bmatrix} 0 & Q \\ -Q^{-T} & 0 \end{bmatrix},
\end{equation*}
then we have
\begin{equation*}
    I+TU=0,\quad V+TUV+T=Q,
\end{equation*}
which implies $T=Q$ leading to contradiction. For the case when the most left-hand side factor is upper triangular, we will also deduce the same contradiction.
\end{proof}
Thus, any symplectic matrix can be factored into no more than $5$ unit triangular symplectic matrices, and 5 is optimal.
\begin{corollary}
$\mathcal{L}_4$ is dense in $\mathcal{L}_5$ in Euclidean topology for $F=\R,\C$.
\end{corollary}
\begin{proof}
Theorem \ref{thm:nonsin_thm} shows that, for any $H\in \mathcal{SP}=\mathcal{L}_5$, there exists a symmetric $S$ such that
\begin{equation*}
    H=\begin{bmatrix} I & \lambda S \\ 0 & I \end{bmatrix}\begin{bmatrix} P_{\lambda} & \star \\ \star & \star \end{bmatrix}
\end{equation*}
holds with a nonsingular $P_{\lambda}$ for all $\lambda\neq0$. Therefore
\begin{equation*}
    H=\lim_{\lambda\rightarrow 0}\begin{bmatrix} I & -\lambda S \\ 0 & I \end{bmatrix}H=\lim_{\lambda\rightarrow 0}\begin{bmatrix} P_{\lambda} & \star \\ \star & \star \end{bmatrix}\in \overline{\mathcal{L}_4},
\end{equation*}
which means $\mathcal{SP}\subset \overline{\mathcal{L}_4}$, consequently $\mathcal{L}_5=\overline{\mathcal{L}_4}$.
\end{proof}
\begin{remark}
For above corollary, we can also regard the symplectic group $\mathcal{SP}$ as an affine algebraic subvariety of $F^{2d\times 2d}$, then $\mathcal{L}_4$ is a Zaraski open subset of $\mathcal{SP}$, where $F$ is a general field. Especially, $\mathcal{L}_4$ is dense in $\mathcal{SP}$ in Euclidean topology if $F$ is chosen as $\R$ or $\C$. 
\end{remark}
We summarize the algorithm of unit triangular factorization as in Algorithm \ref{alg:utf}. The $D$ in step 3 of Algorithm \ref{alg:utf} can also be obtained by Theorem \ref{thm:nonsin_thm}.
\begin{algorithm}
\caption{Unit triangular factorization}
\label{alg:utf}
\begin{algorithmic}
\REQUIRE{$H\in \mathcal{SP}$}
\ENSURE{Symmetric $D,S,T,U,V$ such that $H=\begin{bmatrix} I & D \\ 0 & I \end{bmatrix}\begin{bmatrix} I & 0 \\ S & I \end{bmatrix}\begin{bmatrix} I & T \\ 0 & I \end{bmatrix}\begin{bmatrix} I & 0 \\ U & I \end{bmatrix}\begin{bmatrix} I & V \\ 0 & I \end{bmatrix}$}
\STATE{1. Given $H=\begin{bmatrix}A&\star \\ \star&\star\end{bmatrix}\in \mathcal{SP}$ where $A\in F^{d\times d}$}
\STATE{2. Find the maximal linearly independent rows of $A$ and denote the index set by $\Gamma$}
\STATE{3. Set $D:={\rm diag}(\delta_1,\cdots,\delta_d)$ where $\delta_i=0$ if $i\in\Gamma$ otherwise it is 1}
\STATE{4. Set $\begin{bmatrix} A_{1} & B_{1} \\ A_{2} & B_{2} \end{bmatrix}:=\begin{bmatrix} I & -D \\ 0 & I \end{bmatrix}H$}
\STATE{5. Compute the factorization $A_1=P_1P_2$ where $P_1,P_2$ are symmetric}
\STATE{6. Set $S:=A_{2}A_{1}^{-1}+P_1^{-1}A_1^{-1}-P_1^{-1}$}
\STATE{7. Set $T:=P_1$}
\STATE{8. Set $U:=P_2-P_1^{-1}$}
\STATE{9. Set $V:=A_{1}^{-1}B_1-P_2^{-1}$}
\RETURN $D,S,T,U,V$
\end{algorithmic}
\end{algorithm}

\subsection{Positive definite symplectic matrix}
Denote
\begin{equation*}
\begin{split}
&\mathcal{SPP}=\{H\in\R^{2d\times 2d}|H\ {\rm symmetric,\ positive\ definite\ and\ symplectic}\},\\
&\mathcal{L}_{n} ^2=\{L^TL|L\in\mathcal{L}_{n}\ {\rm with}\ F=\R\},\\
&\mathcal{SPP}^*=\{H\in\C^{2d\times 2d}|H\ {\rm Hermitian,\ positive\ definite\ and\ conjugate\ symplectic}\},\\
&\mathcal{L}_{n}^{*2}= \{L^*L|L\in\mathcal{L}_{n}^{*}\},\\
\end{split}
\end{equation*}
where $\mathcal{L}_n^*$ is the set of conjugate symplectic matrices which are the product of $n$ unit triangular conjugate symplectic matrices. 
\begin{theorem} \label{thm:sppfac_thm}
$\mathcal{SPP}=\mathcal{L}_{3}^2$, $\mathcal{SPP}^*=\mathcal{L}_{3}^{*2}$.
\end{theorem}
\begin{proof}
If $H$ is symmetric positive definite, then all  principal submatrices of $H$ are symmetric positive definite. Denote that $H=\begin{bmatrix} H_{1} & H_{2} \\ H_{3} & H_{4} \end{bmatrix}\in \mathcal{SPP}$, here $H_1$ is positive definite, thus nonsingular. So that in the process of decomposition, instead of going through Theorem \ref{thm:nonsin_thm}/\ref{thm:nonsin_fac}, we just apply the LDU factorization to H, and  get
\begin{equation*}
    H=\begin{bmatrix} I & 0 \\ H_3H_1^{-1} & I \end{bmatrix}\begin{bmatrix} H_1 & 0 \\ 0 & H_1^{-T} \end{bmatrix}\begin{bmatrix} I & H_1^{-1}H_2 \\ 0 & I \end{bmatrix},
\end{equation*}
where $H_3H_1^{-1}$ and $H_1^{-1}H_2$ are symmetric matrices. Moreover, because of the symmetry of $H$, $H_3H_1^{-1}=(H_1^{-1}H_2)^T=H_1^{-1}H_2$. Therefore, every positive definite symplectic matrix can be written as
\begin{equation*}
    H=\begin{bmatrix} I & 0 \\ S & I \end{bmatrix}\begin{bmatrix} P & 0 \\ 0 & P^{-T} \end{bmatrix}\begin{bmatrix} I & S \\ 0 & I \end{bmatrix}
\end{equation*}
with $P$ symmetric positive definite. Then 
\begin{equation*}
\begin{split}
    &\begin{bmatrix} I & 0 \\ S & I \end{bmatrix}\begin{bmatrix} P & 0 \\ 0 & P^{-T} \end{bmatrix}\begin{bmatrix} I & S \\ 0 & I \end{bmatrix}\\
    =&\begin{bmatrix} I & 0 \\ S-I & I \end{bmatrix}\begin{bmatrix} P & P \\ P & P+P^{-1} \end{bmatrix}\begin{bmatrix} I & S-I \\ 0 & I \end{bmatrix}\\
    =&\begin{bmatrix} I & 0 \\ S-I & I \end{bmatrix}\begin{bmatrix} I & -T \\ 0 & I \end{bmatrix}\begin{bmatrix} P+TP+PT+T(P+P^{-1})T & P+T(P+P^{-1}) \\ P+(P+P^{-1})T & P+P^{-1} \end{bmatrix}\begin{bmatrix} I & 0 \\ -T & I \end{bmatrix}\begin{bmatrix} I & S-I \\ 0 & I \end{bmatrix}\\
    =&\begin{bmatrix} I & 0 \\ S-I & I \end{bmatrix}\begin{bmatrix} I & -T \\ 0 & I \end{bmatrix}\begin{bmatrix} I & P+T(P+P^{-1}) \\ P+(P+P^{-1})T & P+P^{-1} \end{bmatrix}\begin{bmatrix} I & 0 \\ -T & I \end{bmatrix}\begin{bmatrix} I & S-I \\ 0 & I \end{bmatrix}\\
    =&\begin{bmatrix} I & 0 \\ S-I & I \end{bmatrix}\begin{bmatrix} I & -T \\ 0 & I \end{bmatrix}\begin{bmatrix} I & 0 \\ P+(P+P^{-1})T & I \end{bmatrix}\begin{bmatrix} I & P+T(P+P^{-1}) \\ 0 & I \end{bmatrix}\begin{bmatrix} I & 0 \\ -T & I \end{bmatrix}\begin{bmatrix} I & S-I \\ 0 & I \end{bmatrix},\\
\end{split}
\end{equation*}
where $T=(-P+\sqrt{P+P^{-1}-I})(P+P^{-1})^{-1}$, which is in fact obtained by solving the quadratic equation
\begin{equation*}
    P+TP+PT+T(P+P^{-1})T=I.
\end{equation*}
So that
\begin{equation}\label{eq:spp}
    H=L^TL,\quad L=\begin{bmatrix} I & P+T(P+P^{-1}) \\ 0 & I \end{bmatrix}\begin{bmatrix} I & 0 \\ -T & I \end{bmatrix}\begin{bmatrix} I & S-I \\ 0 & I \end{bmatrix}.
\end{equation}
Note that both $P+P^{-1}$ and $P+P^{-1}-I$ are positive definite, and any positive definite matrix $M$ has a unique positive definite square root $\sqrt{M}$.

The complex conjugate case is the same.
\end{proof}

\begin{remark}
We actually prove that any positive definite symplectic matrix can be factored as (\ref{eq:spp}) where $L$ has the shape ``upper-lower-upper''. The other case ``lower-upper-lower'' can be obtained by considering nonsingular $H_4$.
\end{remark}

\begin{theorem}
$\mathcal{SPP}\neq\mathcal{L}_{2}^2$, $\mathcal{SPP}^*\neq\mathcal{L}_{2}^{*2}$.
\end{theorem}
\begin{proof}
Suppose that there exist symmetric $S,T$ such that
\begin{equation*}
\begin{bmatrix} I & 0 \\ S & I \end{bmatrix}\begin{bmatrix} I & T \\ 0 & I \end{bmatrix}\begin{bmatrix} I & 0 \\ T & I \end{bmatrix}\begin{bmatrix} I & S \\ 0 & I \end{bmatrix}=\begin{bmatrix} \frac{1}{2}I & 0 \\ 0 & 2I \end{bmatrix},
\end{equation*}
then we have $I+T^2=\frac{1}{2}I$, thus $T^2=-\frac{1}{2}I$ which leads to contradiction. For the case when the most left-hand side factor is upper triangular, we will also deduce the same contradiction.

The complex conjugate case is the same.
\end{proof}
We summarize the factorization algorithm for $\mathcal{SPP}$ as in Algorithm \ref{alg:spp}, and the complex conjugate case is the same ($*$ instead of $T$) and we do not repeat again.
\begin{algorithm}
\caption{Factorization of positive definite symplectic matrix}
\label{alg:spp}
\begin{algorithmic}
\REQUIRE{$H\in \mathcal{SPP}$}
\ENSURE{Symmetric $S,T,U$ such that $H=L^{T}L$,\quad $L=\begin{bmatrix} I & S \\ 0 & I \end{bmatrix}\begin{bmatrix} I & 0 \\ T & I \end{bmatrix}\begin{bmatrix} I & U \\ 0 & I \end{bmatrix}$}
\STATE{1. Given $H=\begin{bmatrix}P & A \\ \star & \star\end{bmatrix}\in \mathcal{SPP}$ where $P,A\in F^{d\times d}$}
\STATE{2. Set $S:=\sqrt{P+P^{-1}-I}$}
\STATE{3. Set $T:=(P-S)(P+P^{-1})^{-1}$}
\STATE{4. Set $U:=P^{-1}A-I$}
\RETURN $S,T,U$
\end{algorithmic}
\end{algorithm}

\subsection{Singular symplectic matrix}
In this subsection, we consider singular symplectic matrices over $K=\R,\C$ \cite{long1990maslov,long1991structure}.
\begin{definition} \label{def:singular_SP}
A symplectic matrix $H\in Sp(d,K)$ is singular if ${\rm det}(H-I)=0$.
\end{definition}
Denote all the singular symplectic matrices by
\begin{equation*}
\mathcal{SPS}=\{H\in Sp(d,K)|H\ {\rm singular} \},
\end{equation*}
then the result \cite[Corollary 4.4]{jin2020unit} can be updated as follows by applying the theorems and lemmas in this paper to its proof.
\begin{theorem}
The set of $2d$-by-$2d$ singular symplectic matrices is
\begin{equation*}
\begin{split}
\mathcal{SPS}=&\Bigg\{Q\begin{bmatrix}I&0\\\begin{bmatrix}0&0\\0&S_{6}\end{bmatrix}&I\end{bmatrix}\begin{bmatrix}I&S_{5}\\0&I\end{bmatrix}\cdots
\begin{bmatrix}I&0\\\begin{bmatrix}0&0\\0&S_{2}\end{bmatrix}&I\end{bmatrix}\begin{bmatrix}I&S_{1}\\0&I\end{bmatrix}Q^{-1} \\
&\Bigg|S_{2i-1}\in K^{d\times d}\ {\rm symmetric},\ S_{2i}\in K^{(d-1)\times(d-1)}\ {\rm symmetric},\ Q\ {\rm symplectic}\Bigg\}.
\end{split}
\end{equation*}
One can express $Q$ as the product of $5$ unit triangular symplectic matrices if needed.
\end{theorem}

\subsection{Unconstrained optimization}
An optimization problem with symplectic constraint is in the following form
\begin{equation}
\label{eq:ConOP}
\min\limits_{X \in \R^{2d \times 2d}}f(X),\quad s.t.\ X^TJX=J.
\end{equation}
There have been many works on optimization on the real symplectic group \cite{birtea2020optimization,fiori2011solving,wang2018riemannian}, in which one performs optimization by considering the gradients along the manifold. \cite{jin2020unit} has pointed out that the unit triangular factorization provides an approach to the symplectic optimization from a new perspective, i.e., optimizing in a higher dimensional unconstrained parameter space. We define the map $Pa$ for extracting the lower triangular parameters as $Pa(S)=(s_{11},s_{21},s_{22},s_{31},\cdots,s_{dd})$, where $S=(s_{ij})\in \R^{d\times d}$, $S^{T}=S$. Take symmetric $S_{1},S_{2},\cdots,S_{4}\in \R^{d\times d}$ and a vector $v\in\R^d$, then
\begin{equation*}
H(v,Pa(S_{1}),\cdots,Pa(S_{4}))=\begin{bmatrix} I & {\rm diag}(v) \\ 0 & I \end{bmatrix}\begin{bmatrix} I & 0 \\ S_{4} & I \end{bmatrix}\begin{bmatrix} I & S_{3} \\ 0 & I \end{bmatrix}\begin{bmatrix} I & 0 \\ S_{2} & I \end{bmatrix}\begin{bmatrix} I & S_{1} \\ 0 & I \end{bmatrix} \\
\end{equation*}
can represent any symplectic matrix when $v,Pa(S_{1}),\cdots,Pa(S_{4})$ vary.
Problem (\ref{eq:ConOP}) is equivalent to
\begin{equation*}
\min\limits_{\substack{v\in\R^d \\ Pa(S_{i})\in\R^{\frac{d(d+1)}{2}}}}f(H(v,Pa(S_{1}),\cdots,Pa(S_{4}))),
\end{equation*}
which is indeed an unconstrained optimization problem with parameters of $2d^2+3d$, that has the same quadratic term as the dimension of Lie group $Sp(d,\R)$, i.e., $2d^2+d$. The optimal unit triangular factorization significantly reduces the number of parameters compared to \cite{jin2020unit}. In fact, this method has been utilized in recent work \cite{jin2020sympnets}. In such a case, the unit triangular factorization-based optimization can be implemented directly within the deep learning framework and performs well, while the traditional Riemannian-steepest-descent approach faces challenges.

\section{Conclusions}
In this work, the optimal unit triangular factorization of symplectic matrices is given. We prove that any symplectic matrix can be factored into no more than 5 unit triangular symplectic matrices, moreover, 5 is the optimal number. We also show the corresponding improved conclusions for structured subsets of symplectic matrices, i.e., positive definite symplectic matrices and singular symplectic matrices. This factorization also provides an unconstrained optimization method on real symplectic group. 

\bibliographystyle{abbrv}
\bibliography{references}

\end{document}